\documentclass[preprint,12pt]{elsarticle} 
\usepackage{graphicx}
\usepackage{amsmath}
\usepackage{amsthm}
\usepackage{geometry}
\usepackage{array}
\usepackage{amsfonts}

\newtheorem{thm}{Theorem}

\newtheorem{lem}{Lemma}

\newtheorem{ex}{Example}

\journal{}

\begin{document}

\begin{frontmatter} 

\title{An Integer Linear Programming Formulation for the Convex Dominating Set Problems}

\author[mi]{Jozef J. Kratica}
\ead{jkratica@mi.sanu.ac.rs}

\author[matf]{Vladimir Filipovi\'c}
\ead{vladofilipovic@hotmail.com}

\author[banj]{Dragan Mati\'c}
\ead{dragan.matic@pmf.unibl.org}

\author[matf]{Aleksandar Kartelj}
\ead{aleksandar.kartelj@gmail.com}

\address[mi]{Mathematical Institute,
Serbian Academy of Sciences and Arts, 
Kneza Mihaila 36/III, 11000 Belgrade, Serbia}

\address[matf]{Faculty of Mathematics, University of Belgrade, Studentski
trg 16/IV, 11000 Belgrade, Serbia}

\address[banj]{Faculty of Mathematics and Natural Sciences, University of Banjaluka, 
Mladena Stojanovi\'ca 2, Banjaluka, Bosnia and Herzegovina}


\begin{abstract}
Due to their importance in practice, dominating set problems in graphs have been greatly studied
in past and different formulations of these problems are presented in literature. 
This paper's focus is on two problems: weakly convex dominating set problem (WCVXDSP) and convex dominating set problem (CVXDSP). 
It introduces two integer linear programming (ILP) formulation for CVXDSP and one ILP mode for WCVXDSP, 
as well as proof for equivalency between ILP models for CVXDSP. 
The proof of correctness for all introduced ILP formulations is provided by showing that optimal solution to 
the each ILP formulation is equal to the optimal solution of the original problem.
\end{abstract}

 
\end{frontmatter} 

\section{Introduction}
Dominating set problems have been used in wireless networks to address issues such as: media access, routing, power management, and topology control. Another real-life application of these problems is using them to gain insight into social networks dynamics.
Different variants of dominating set problems are considered in literature. Paper \cite{cockayne1978domination} provides detailed introduction into domination problems on undirected graphs, while \cite{yu2013connected} gives the review of the applications of connected dominating sets in wireless network topology design. 
In \cite{chellali2012k}, the authors survey the results on the concept of $k$-domination which can be viewed as a generalization of the domination in graph. Another comprehensive study on domination related problems is made in \cite{goddard2013independent} where the authors present selected results on independent domination in graphs.

In this paper, the weakly convex dominating sets and convex dominating sets are the focus, therefore, in the following text, all relevant concepts and notation are provided.

Let $G = (V,E)$ be a connected undirected graph without loops and parallel edges. With $\mathcal{N}(v)$ the set of all vertices adjacent to $v$ is denoted. 
A vertex $u\in V$ is \textit{dominated} by a set $D\subseteq V$ if either $u$ itself or one of its neighbors is in $D$, i.e.
$u\in D \vee (\exists t \in N(u)) t \in D$.
A set $V'\subseteq V$ is a \textit{dominating set} in $G$ if every vertex in $V$ is dominated by $V'$. In other words, set $V'\subset V$ is a dominating set in $G$ if 
for each vertex $u\in V\setminus V'$ exists a vertex $v\in V'$ that is adjacent to $u$.  

The \textit{domination number} of graph $G$, denoted as $\gamma(G)$ is defined as the set of minimum cardinality among all dominating sets of $G$. 
  
\begin{ex} \label{ex1} 
  Given the graph $G(V,E)$ (Figure \ref{fig_ncvxdsp_example}), with vertices $V=\{1, 2, 3, 4, 5, 6, 7\}$ and edges $E=\{(1,2), (1,4), (1,5), (2,3), (3,4), (3,7)\}$, 
	it can be observed that the set $D=\{1,3\}$ is the smallest dominating set in $G$, therefore, $\gamma(G)=2$.
\end{ex}

\begin{figure}[h]
\centering
\includegraphics[width=6cm,height=4.2cm]{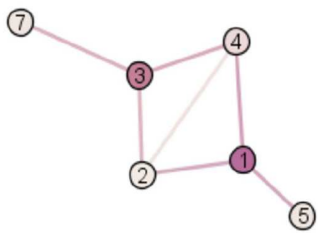} 
\caption{Dominating set}
\label{fig_ncvxdsp_example}
\end{figure}

Terms \emph{dominating set} and \emph{domination number} of a graph $G$ were introduced by O. Ore in 1962 \cite{ore62}. 
Various types of domination and graph domination numbers have been analyzed across multiple classes of graphs. This paper deals with convex domination and weakly convex domination which can be introduced  only on graphs where distance among connected vertices is introduced. 

Let $d_G(u,v)$ denote the distance between the vertices $u$ and $v$. Note that distance refers to the length of the shortest $u-v$ path in $G$.
A set $W\subseteq V$ is a \textit{weakly convex set} if for every two vertices $u$ and $v$ from $W$, it holds that $d(u,v)$ in $\langle W\rangle$ is equal to $d(u,v)$ 
in $G$, where $\langle W\rangle$  is a subgraph of $G$ induced by $W$. 

Equivalently, a set $W\subseteq V$ is a weakly convex in $G$ if for every two vertices $u,v\in W$, there exists at least one shortest $u-v$ path (in $G$), whose vertices belong to $W$.  

A set $W$ is a \textit{weakly convex dominating set} if it is weakly convex and dominating set. The \textit{weakly convex domination number}, denoted as 
$\gamma_{wcvx}(G)$ of a graph $G$ is a weakly convex dominating set of the smallest cardinality. 
The \textit{weakly convex dominating set problem} (WCVXDSP) is an optimization problem of determining a weakly convex dominating set of the smallest cardinality. This set is also known as \textit{minimal weakly convex dominating set}. 

\begin{ex} \label{ex2} 
  Dominating set $D$ in graph $G$, defined in  Example \ref{ex1} and shown in Figure \ref{fig_ncvxdsp_example}, is not a weakly convex, because the shortest path between vertices $1$ and $3$ does not pass through $D$.
Moreover, there are no other subsets of cardinality 1 or 2 that are weakly convex dominating sets. 
Therefore, for that graph, $\gamma_{wcvx}(G)>2$. Since $Y=\{1,2,3\}$ and $Z=\{1,3,4\}$ are weakly convex dominating in $G$, then $\gamma_{wcvx}(G)=3$.
\end{ex}

\begin{ex} \label{ex3} 
 In the graph $G(V,E)$, with vertices $V=\{1, 2, 3, 4, 5, 6, 7, 8\}$ and with edges $E=\{(1,2), (1,4), (1,5), (2,3), (3,4), (3,7), (4,8)\}$, 
as shown in Figure \ref{fig_wcvxdsp_example}, it can be concluded that set $D=\{1, 3, 4\}$ is the dominating set in graph $G$. It can be also seen that there is no dominating set in $G$ of cardinality $2$. Therefore, $\gamma(G)=3$.

For every pair $u, v$ of vertices in $D$, exists the shortest $u-v$ path (in $G$), whose vertices belong to $D$. Therefore, $D$ is weakly convex dominating
set in $G$ and $\gamma_{wcvx}(G)=3$.
\end{ex}

\begin{figure}[h]
\centering
\includegraphics[width=6cm,height=4.2cm]{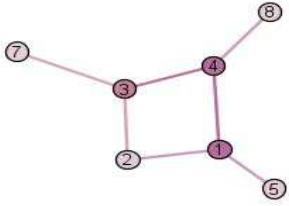} 
\caption{Weakly convex dominating set}
\label{fig_wcvxdsp_example}
\end{figure}

A set $X\subseteq G$ is a \textit{convex set} if for every two vertices $u$ and $v$ from $X$, every shortest $u-v$ path (in $G$) also belongs to $X$.
A set $X$ is a \textit{convex dominating set} if it is both the convex and the dominating set. 

The \textit{convex domination number}, denoted as $\gamma_{cvx}(G)$ of a graph $G$
is a convex dominating set of the smallest cardinality in $G$. Therefore, \textit{convex dominating set problem} (CVXDSP) is an optimization problem of determining a convex 
dominating set of the smallest cardinality. This set is also known as \textit{minimal convex dominating set}. 

\begin{ex} \label{ex4} 

Weakly convex dominating set $W=\{1, 3, 4\}$ in graph $G$, defined in Example \ref{ex3} and shown in Figure \ref{fig_wcvxdsp_example}, is not convex, 
because there is a shortest path between vertices $1$ and $3$ (path $1-2-3$), which contains vertex $2$ that does not belong to $W$. 
Moreover, since $\gamma(W)=3$ and there are no other convex dominating sets of cardinality 3, $\gamma_{cvx}(G)>3$. Since $\{1, 2, 3, 4\}$ is a convex dominating set in $G$, then $\gamma_{cvx}(G)=4$.      
\end{ex}

\begin{ex} \label{ex5} 
 In the graph $G(V,E)$, where vertices are $V=\{1, 2, 3, 4, 5, 6, 7, 8\}$ and edges are $E=\{(1,2), (1,3), (1,4), (1,5), (2,3), (3,4), (3,7), (4,8)\}$ 
(given in Figure \ref{fig_cvxdsp_example}), it can be concluded that set $D=\{1, 3, 4\}$ is the dominating set in graph $G$. There is no dominating set in $G$ that consists of two elements and $D$, so $\gamma(G)=3$.

For every pair $u, v$ of vertices in $D$, exists the shortest  $u-v$ path (in $G$), whose vertices belong to $D$. Therefore, $D$ is weakly convex dominating set in $G$ and $\gamma_{wcvx}(G)=3$.

Furthermore, for every pair $u, v$ of vertices in $V$, all the shortest $u-v$ paths (in $G$) are going through vertices in $D$, so $D$ is convex dominating set in $G$ and $\gamma_{cvx}(G)=3$.
\end{ex}

\begin{figure}[h]
\centering
\includegraphics[width=6cm,height=4.2cm]{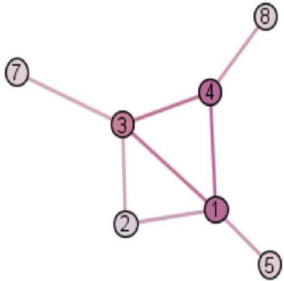} 
\caption{Convex dominating set}
\label{fig_cvxdsp_example}
\end{figure}

\subsection{Related work}
Jerzy Topp, from Gdansk University of Technology coined the term \emph{convex domination number} in 2002. 
The decision variants of WCVXDSP and CVXDSP are NP-complete even for bipartite and split graphs \cite{rac04}. Therefore, problems of determining both the weakly convex dominating and convex dominating sets of minimal cardinality are NP-hard.

In paper \cite{lem04}, the author studied relations between $\gamma_{wcvx}$ and $\gamma_{cvx}$ for some classes of cubic graphs. In these graphs, convex domination number is equal to the domination number. 

Based on the fact that every convex dominating set is a weakly convex dominating set and every weakly convex dominating set is a dominating set, the following lemma is proposed in \cite{lem04}:
\begin{lem}\label{lemma:lom04}\cite{lem04} For any connected graph $G$
$$\gamma(G)\leq \gamma_{wcvx}(G)\leq \gamma_{cvx}(G)$$
\end{lem}

The authors of paper \cite{jan10} presented multiple bounds for the weakly convex domination number and the convex domination number.

Closed formulas for weakly convex and convex domination numbers of a torus are proposed in \cite{rac10}.

The edge subdivision influence on the convex domination number is discussed in \cite{det12}. 
In that paper it is shown that, in general, the convex domination number can be arbitrarily increased and decreased by an edge subdivision.
Study of weakly convex domination subdivision number and its upper bounds is presented in \cite{dettlaff2016weakly}. 

Nordhaus--Gaddum type results for the weakly convex domination number and convex domination number are covered in \cite{lem10} and \cite{lem12}, respectively.

\section{Integer linear programming formulation}

Let $G$ = ($V$,$E$) be a simple connected undirected graph, with
$V=\{1,2,\dots,n\}$ and $|E|=m$. The length
$d(u,v)$ of a shortest $u-v$ path for all $u,v \in V$ can be calculated using any
shortest path algorithm. 

Decision variable $x_i$ indicates whether vertex $i$ belongs to a convex
dominating set $X$.

\begin{equation}\label{defx}
x_i  =\begin{cases}
1, & i \in X\\
0, & i \notin X
\end{cases}
\end{equation}

The integer linear programming model of the convex dominating set problem can now be
formulated as:

\begin{equation}\label{obj}
\min \sum\limits_{i = 1}^n {x_i }
\end{equation}

subject to:

\begin{equation}\label{c1}
x_i + \sum\limits_{j \in \mathcal{N}(i)} {x_j}  \ge 1\quad \quad
\quad \quad \quad \quad \quad 1 \le i \le n
\end{equation}

\begin{equation}\label{c2}
\begin{split}
x_k - x_i - x_j  \ge -1  \quad \quad &   1 \le i <j \le n, \\
                         &  k \in V, \\
                         &  d(i,k)+d(k,j)=d(i,j)
\end{split}
\end{equation}

\begin{equation}\label{c3}
x_i  \in \{ 0,1\} \quad \quad \quad \quad \quad 1 \le i \le n
\end{equation}

Constraints (\ref{obj}) represent the objective function, while constraints (\ref{c1}) ensure domination. 
Convexity is enforced by constraints (\ref{c2}) and the binary nature of the decision variables is provided by (\ref{c3}). 
It should be noted that the ILP model (\ref{obj})-(\ref{c3}) has $n$ binary variables and $O(n \cdot m)$ constraints.

Constraints (\ref{c2}) can be replaced with new ones (\ref{c2a}). 
\begin{equation}\label{c2a}
\begin{split}
x_k - x_i - x_j  \ge -1  \quad \quad &   1 \le i <j \le n, \\
                         &  k \in \mathcal{N}(i) \cup \mathcal{N}(j), \\
                         &  d(i,k)+d(k,j)=d(i,j)
\end{split}
\end{equation}

\begin{lem}\label{l1} Conditions (\ref{c3}) and (\ref{c2a}) imply that (\ref{c2}) holds. \end{lem}

\begin{proof} 
Let $i,j,k \in V$ and $d(i,k)+d(k,j)=d(i,j)$. Without loss of generality, it can be assumed that $i<j$.

There are two cases:

\em{x:} $x_j = 0$. From (\ref{c3}) holds $x_i, x_k \in \{0,1\}$, so $x_k \geq 0$ and $-x_i \geq -1$. Since $x_j = 0$
is obviously that $x_k - x_i - x_j \geq -1$.

\em{Case 2:} $x_j = 1$. Since $G$ is connected and $d(i,k)+d(k,j)=d(i,j)$  then vertex $k$ belongs to a shortest path from vertex
$i$ to vertex $j$. 
Let us denote the shortest path with $p_0=i, p_1, \ldots, p_q=k, \ldots, p_r=j$, where 
$ r=d(i,j) \wedge (\forall l \in \{0,\ldots,r-1\}) p_{l+1} \in \mathcal{N}(p_l)$. 
Due to the fact that $p_{l+1} \in \mathcal{N}(p_l)$ and (\ref{c2a}), it holds $x_{p_{l+1}} - x_{p_l} - x_j \geq -1$ for $l=0,\ldots, q-1$, 
which is equivalent to $x_{p_{l+1}} - x_{p_l}  \geq x_j-1$. Summing those equations, we obtain $\sum\limits_{l=0}^{q-1} (x_{p_{l+1}} - x_{p_l}) \geq \sum\limits_{l=0}^{q-1}(x_j-1)$, 
implying $x_{p_{q}} - x_{p_{0}} \geq q \cdot (x_j-1)$, so $x_k - x_i \geq q \cdot (x_j-1)$.
Since $x_j = 1$, then $q \cdot (x_j-1) = 0 = x_j-1 $. Therefore, from two previous sentences it holds $x_k - x_i \geq x_j-1$.

In both cases, statement (\ref{c2}) is proven.
\end{proof}

\begin{thm}\label{tr} Optimal solution value of model (\ref{obj})-(\ref{c3}) is equal to optimal solution value of model (\ref{obj}), (\ref{c1}), 
(\ref{c2a}), (\ref{c3}). \end{thm}

\begin{proof}  

It is obvious that constraints (\ref{c2a}) are a subset of constraints (\ref{c2}), so the feasible solution space of 
 (\ref{c1}), (\ref{c2a}), (\ref{c3}) contains the feasible solution space of (\ref{c1})-(\ref{c3}). 
Since the objective function (\ref{obj}) is the same for both models, 
it can be concluded that the value of objective function of model (\ref{obj}), (\ref{c1}), 
(\ref{c2a}), (\ref{c3}) is less or equal to the value of objective function of model  (\ref{obj})-(\ref{c3}). 

From Lemma \ref{l1}, the inverse statement holds: the feasible solution space of 
 (\ref{c1}), (\ref{c2a}), (\ref{c3}) is the subset of the feasible solution space of (\ref{c1})-(\ref{c3}). 

Therefore, the optimal solution value of model (\ref{obj})-(\ref{c3}) is equal to the optimal solution value of model 
(\ref{obj}), (\ref{c1}),  (\ref{c2a}), (\ref{c3}).
\end{proof}

The following theorem shows that optimal solution of (\ref{obj})-(\ref{c3}) defines a minimal convex dominating set $X$ of $G$, and vice versa.

\begin{thm}\label{tc}Set $X$ is a minimal convex dominating set of $G$ if and only if constraints (\ref{obj})-(\ref{c3})
are satisfied. \end{thm}

\begin{proof} ($\Rightarrow$) Let $X$ be the minimal convex dominating set of $G$ and decision variables $x$ are defined by (\ref{defx}). 
Constraints (\ref{c3}) about binary nature of $x$ variables are trivially satisfied from the definition (\ref{defx}). 

Since $X$ is the dominating set, then $(\forall i \in V)(\exists j \in \mathcal{N}(i)) (i \in X \vee j \in X)$ imply that 
$(\forall i \in V)(\exists j \in \mathcal{N}(i)) (x_i = 1 \vee x_j = 1)$, which means 
$(\forall i \in V) x_i + \sum\limits_{j \in \mathcal{N}(i)} {x_j}  \ge 1$. In that way, constrains (\ref{c1}) are satisfied.

For $i, j \in V$, without loss of generality it can be assumed $1 \le i < j \le n$. The following two cases are possible:

{\em Case 1.} $i \notin X \vee j \notin X$. In this situation, $x_i = 0 \vee x_j=0$ and $x_i, x_j \in \{0,1\}$ then $x_i + x_j \le 1 $ implying 
$-x_i - x_j \ge -1$. Since $(\forall k \in V) x_k \ge 0 $, then $x_k -x_i - x_j \ge -1$. 

{\em Case 2.} $i \in X \wedge j \in X$. Let $d(i,k)+d(k,j)=d(i,j)$, which means that $k$ belongs to an $i-j$ shortest path. Since  
$X$ is a convex set, every vertex that belongs to that $i-j$ shortest path has to belong to set $X$, so  $k \in X$. Therefore,
$x_k = x_i = x_j = 1 $ implying $x_k - x_i - x_j = -1 \ge -1 $. 

In both cases constraints (\ref{c2}) are satisfied.

From the definition (\ref{defx}) it holds $|X| = \sum\limits_{i = 1}^n {x_i } $. 
Since decision variables $x$ represent a feasible solution of ILP model 
(\ref{obj})-(\ref{c3}), its optimal solution has to be less or equal to $|X|$.

($\Leftarrow$) Let $X=\{i\,|\, x_i = 1\}$. Since variables $x_i$ are binary, from (\ref{c1}) it holds 
$(\forall i) (x_i=1 \vee \sum\limits_{j \in \mathcal{N}(i)} {x_j} \ge 1) \Rightarrow $
$(\forall i) (x_i=1 \vee (\exists j \in \mathcal{N}(i)) x_j=1) $. So, $(\forall i) (i \in X \vee (\exists j \in \mathcal{N}(i)) j \in X) $
and therefore $X$ is the dominating set of $G$.

Let $i, j \in X$, so $x_i = x_j = 1$. Further, without loss of generality, let $i<j$. Let vertex $k \in V$ which belongs to an $i-j$ shortest path, so $d(i,k)+d(k, j)=d(i,j)$. 
Applying (\ref{c2}) we have that $x_k - x_i - x_j \geq -1$, which is equivalent to $x_k \geq x_i + x_j -1$. Since $x_i = x_j = 1$, then $x_k \geq 1$. 
From binary nature of variables $x_k$, given by (\ref{c3}), $x_k = 1$ holds, meaning $k \in X$. Therefore, for each pair $i, j \in X$ and any vertex $k$ belonging to $i-j$ shortest path, $k \in X$ holds, so $X$ is a convex set.

Since $X$ is proven to be convex domination set, minimal convex domination number has to be less or equal to $\sum\limits_{i = 1}^n {x_i } $.
\end{proof}

If constraints (\ref{c2}) are replaced by 

\begin{equation}\label{c2w}
- x_i - x_j  + \sum\limits_{\begin{array}{*{20}{c}}
{k \in {\mathcal{N}(j)}}\\
{d(i,k) + d(k,j) = d(i,j)}
\end{array}} {{x_k}} \ge -1  \quad \quad  1 \le i <j \le n
\end{equation}

then (\ref{obj}), (\ref{c1}), (\ref{c2w}) and (\ref{c3}) represents the ILP formulation for weak convex 
domination problem. This fact is formulated in Theorem \ref{twc}.

\begin{thm}\label{twc} Set $W$ is a minimal weak convex dominating set of $G$ if and only if constraints (\ref{obj}), (\ref{c1}), (\ref{c3}) and (\ref{c2w})
are satisfied. \end{thm}
\begin{proof} ($\Rightarrow$) Let $W$ be the minimal weak convex dominating set of $G$ and decision variables $x$ are defined by: 

\begin{equation}
\nonumber
x_i  =\begin{cases}
1, & i \in W\\
0, & i \notin W.
\end{cases}
\end{equation}
Again, constraints (\ref{c3}) about binary nature of $x$ variables are trivially satisfied by their definition. 
Since $W$ is the dominating set, then, in the same way as in proof of Theorem \ref{tc}, constrains (\ref{c1}) are satisfied.
Also, from the definition of decision variables $x$ it holds $|W| = \sum\limits_{i = 1}^n {x_i } $. Similarly, as in the proof of Theorem \ref{tc}, it can be shown that the optimal solution value of ILP model 
(\ref{obj}), (\ref{c1}), (\ref{c3}) and (\ref{c2w}) is less or equal to $|W|$.

It remains to be proven that constraints (\ref{c2w}) hold.
For $i, j \in V$, without loss of generality it can be assumed $1 \le i < j \le n$. 
Similarly as in proof of Theorem \ref{tc} we have two possible cases.

{\em Case 1.} $i \notin W \vee j \notin W$.\\
Since $x_i = 0 \vee x_j=0$ and binary nature of variables $x$ we have $x_i + x_j \le 1$ implying 
$-x_i - x_j \ge -1$. The following implications hold:

\begin{equation} 
\nonumber
(\forall k \in V) x_k \ge 0 \Rightarrow \sum\limits_{\begin{array}{*{20}{c}}
{k \in {\mathcal{N}(j)}}\\
{d(i,k) + d(k,j) = d(i,j)}
\end{array}} {{x_k}} \ge 0 
\end{equation}

\begin{equation} 
\nonumber
 \Rightarrow  - x_i - x_j  + \sum\limits_{\begin{array}{*{20}{c}}
{k \in {\mathcal{N}(j)}}\\
{d(i,k) + d(k,j) = d(i,j)}
\end{array}} {{x_k}} \ge -1
\end{equation}

{\em Case 2.} $i \in W \wedge j \in W$.\\
Let $s=d(i,j)$. Because $i<j$, it holds that $s \geq 1$. Since $W$ is the weak convex dominating set,
then there exists (at least one) shortest $i-j$ path of length $s$
named $i=p_0,p_1,...,p_s=j$ whose all vertices belong to $W$.
Vertex $p_{s-1} \in W$, therefore it holds that $x_{p_{s-1}}=1$. 
Since $p_{s-1}$ is the second to last vertex in the $i-j$ shortest path, then
 $d(i,p_{s-1})=s-1 \wedge d(p_{s-1},j)=1$ implying $d(i,p_{s-1})+d(p_{s-1},j)=s=d(i,j)$.
Vertex $p_{s-1} \in {\mathcal{N}(j)}$ because $d(p_{s-1},j)=1$. Then,

\begin{equation}
\nonumber
- x_i - x_j  + \sum\limits_{\begin{array}{*{20}{c}}
{k \in {\mathcal{N}(j)}}\\
{d(i,k) + d(k,j) = d(i,j)}
\end{array}} {{x_k}} = 
\end{equation}

\begin{equation}
\nonumber
- x_i - x_j  + x_{s-1} + \sum\limits_{\begin{array}{*{20}{c}}
{k \ne p_{s-1} \wedge k \in {\mathcal{N}(j)}}\\
{d(i,k) + d(k,j) = d(i,j)}
\end{array}} {{x_k}}  \ge -1,
\end{equation}
since $x_i = x_j = x_{s-1} = 1$ and all decision variables $x$ are
binary (non-negative).

Therefore, it is proven that constraints (\ref{c2}) are satisfied for both cases. 

($\Leftarrow$) Let $W=\{i\,|\, x_i = 1\}$. In the same way as in proof of Theorem \ref{tc},
$W$ is the dominating set of $G$. We need to prove that $W$ is a weak convex set,
i.e. that for each $i,j \in W$ there exists at least one shortest $i - j$ path whose vertices belong to $W$.

This property will be proven using the mathematical induction by distance $d(i,j)$.\\
{\em Step 1. }  $d(i,j)=1$. \\
Let $i,j \in W$. Since $d(i,j)=1$ then $i-j$ shortest path has length 1, i.e. $i$ and $j$ are the only vertices in that
shortest path, so both vertices from the $i-j$ shortest path ($i$ and $j$) belong to $W$. \\
{\em Step 2. }  $d(i,j)=s$. \\
Lets assume inductive hypothesis, that for each pair of vertices in $W$ with distance $s$,
there is at least one shortest path whose all vertices also belong to $W$. 
We must prove this fact also holds for all pairs of vertices with distance $s+1$.
Let $i,j \in W$ with $d(i,j)=s+1$. Without loss of generality it can be assumed that $i<j$.
From $i,j \in W$ it is implied that $x_i = x_j=1$. Since constraints (\ref{c2w}) must be satisfied, 
it must also hold that:

\begin{equation}
\nonumber
 \sum\limits_{\begin{array}{*{20}{c}}
{k \in {\mathcal{N}(j)}}\\
{d(i,k) + d(k,j) = d(i,j)}
\end{array}} {{x_k}} \ge 1.
\end{equation}

Due to the binary nature of $x$ then $(\exists k) (x_k=1 \wedge k \in \mathcal{N}(j) \wedge d(i,k)+d(k,j)=d(i,j))$, implying $(\exists k) (k \in W \wedge d(k,j)=1 \wedge d(i,k)+d(k,j)=s+1)$. 
Consequently, $(\exists k) (k \in W \wedge d(i,k)=s)$ and vertex $k$ belongs to an $i-j$ shortest path. 

From inductive hypothesis stating that both vertices $i$ and $k$ belong to $W$ with distance $s$, we have
that there exists (at least one) shortest $i-k$ path 
named $i=p_0,p_1,...,p_s=k$ whose all vertices belong to $W$.
Since $d(i,j)= d(i,k) + d(k,j) = s+1$ we can construct the shortest
$i-j$ path $i=p_0,p_1,...,p_s=k,p_{s+1}=j$ whose all vertices belong to $W$,
so {\em Step 2. } is proven.

Since $W$ is proven to be weakly convex domination set, minimal weakly convex domination number has to be less or equal to $\sum\limits_{i = 1}^n {x_i } $.
\end{proof}

\section{Conclusions}

In this work, we studied the weakly convex dominating set problem and convex dominating set problem. 
For convex dominating set problem, two ILP formulations were introduced and their related ILP models are proven to be equal. For weakly convex dominating set problem, one ILP formulation is introduced. 
It is also confirmed that all ILP formulations are correct since their optimal solutions are equal to the optimal solutions of the starting problems.

This work can be extended in several ways. For example, the proposed models can be experimentally tested across different classes of graph instances by applying exact ILP solvers. 
It might also be interesting to design a meta-heuristic algorithm to deal with large-scale instances. 



\begin{thebibliography}{10}

\bibitem{cockayne1978domination}
E.~Cockayne, Domination of undirected graphsΓÇöa survey, in: Theory and
  Applications of Graphs, Springer, 1978, pp. 141--147.

\bibitem{yu2013connected}
J.~Yu, N.~Wang, G.~Wang, D.~Yu, Connected dominating sets in wireless ad hoc
  and sensor networks--a comprehensive survey, Computer Communications 36~(2)
  (2013) 121--134.

\bibitem{chellali2012k}
M.~Chellali, O.~Favaron, A.~Hansberg, L.~Volkmann, k-domination and
  k-independence in graphs: A survey, Graphs and Combinatorics 28~(1) (2012)
  1--55.

\bibitem{goddard2013independent}
W.~Goddard, M.~A. Henning, Independent domination in graphs: A survey and
  recent results, Discrete Mathematics 313~(7) (2013) 839--854.

\bibitem{ore62}
O.~Ore, Theory of graphs, American Mathematical Society Colloquium
  Publications, Vol. XXXVIII, American Mathematical Society, Providence, R.I.,
  1962.

\bibitem{rac04}
J.~Raczek, N{P}-completeness of weakly convex and convex dominating set
  decision problems, Opuscula Math. 24~(2) (2004) 189--196.

\bibitem{lem04}
M.~Lema{\'n}ska, Weakly convex and convex domination numbers, Opuscula Math.
  24~(2) (2004) 181--188.

\bibitem{jan10}
T.~Janakiraman, P.~Alphonse, Weak convex domination in graphs, International
  Journal of Engineering Science, Advanced Computing and Bio-Technology 1~(1)
  (2010) 1--13.

\bibitem{rac10}
J.~Raczek, M.~Lema{\'n}ska, A note on the weakly convex and convex domination
  numbers of a torus, Discrete Appl. Math. 158~(15) (2010) 1708--1713.

\bibitem{det12}
M.~Dettlaf, L.~M., Influence of edge subdivision on the convex domination
  number, Australasian Journal of Combinatorics 53 (2012) 19--30.

\bibitem{dettlaff2016weakly}
M.~Dettlaff, S.~Kosary, M.~Lema{\'n}ska, S.~M. Sheikholeslami, Weakly convex
  domination subdivision number of a graph, Filomat 30~(8) (2016) 2101--2110.

\bibitem{lem10}
M.~Lema{\'n}ska, Nordhaus-{G}addum results for weakly convex domination number
  of a graph, Discuss. Math. Graph Theory 30~(2) (2010) 257--263.

\bibitem{lem12}
M.~Lema{\'n}ska, J.~A. Rodr{\'{\i}}guez-Vel{\'a}zquez, I.~Gonzalez~Yero,
  Nordhaus-{G}addum results for the convex domination number of a graph,
  Period. Math. Hungar. 65~(1) (2012) 125--134.

\end{thebibliography}
\end{document}